\documentclass[11pt]{article}
\usepackage[utf8]{inputenc}
\usepackage{amsmath, amssymb, amsthm}
\usepackage{graphicx}
\usepackage{geometry}
\usepackage{hyperref}
\usepackage{enumitem}
\usepackage{titlesec}
\usepackage{cite}
\usepackage{mathrsfs}
\usepackage{subcaption}
\usepackage{caption}
\usepackage{comment}
\usepackage{listings}
\usepackage{xcolor}
\usepackage{setspace}
\usepackage[nopatch=footnote]{microtype}

\lstdefinelanguage{Lean}{
  keywords={def, theorem, lemma, example, have, show, from, by, fun, let, in,
            if, then, else, match, with, import, namespace, end, variable,
            universe, sort, Prop, Type, structure, where, noncomputable,
            section, open, calc, forall, exists, return},
  sensitive=true,
  comment=[l]{--},
  morecomment=[s]{/-}{-/},
  morestring=[b]"
}

\titleformat{\section}{\large\bfseries}{\thesection}{1em}{}
\titleformat{\subsection}{\normalsize\bfseries}{\thesubsection}{1em}{}
\titleformat{\subsubsection}{\normalsize\itshape}{\thesubsubsection}{1em}{}

\newtheorem{theorem}{Theorem}[section]
\newtheorem{lemma}[theorem]{Lemma}

\newtheorem{corollary}[theorem]{Corollary}
\theoremstyle{definition}

\newtheorem{remark}[theorem]{Remark}

\title{Formalizing the Classical Isoperimetric Inequality\\in the Two-Dimensional Case}
\author{Miraj Samarakkody}
\date{\today}

\begin{document}

\maketitle

\begin{abstract}
We present a formal verification of the classical isoperimetric inequality in the plane
using the Lean~4 proof assistant and its mathematical library Mathlib.  We follow Adolf
Hurwitz's analytic approach to establish the inequality $L^2 \geq 4\pi A$, which states
that among all simple closed curves of a given perimeter $L$, the circle uniquely
maximizes the enclosed area $A$.  The formalization proceeds in two phases.  In the first
phase, we establish the Fourier-analytic foundations required by Hurwitz's approach: we
formalize orthogonality relations for trigonometric functions over $[-\pi, \pi]$,
Parseval's theorem for classical Fourier series, uniform convergence of Fourier partial
sums via the Weierstrass M-test, term-by-term differentiability, and Wirtinger's
inequality.  In the second phase, we carry out Hurwitz's proof itself: working with simple
closed C\textsuperscript{1} curves given in arc-length parametrization, we reparametrize
over $[0,2\pi]$, establish the shoelace area formula, apply integration by parts, invoke
the AM--GM inequality, apply Wirtinger's inequality, and use the arc-length constraint to
derive the bound $A \leq L^2/(4\pi)$.  We discuss the key formalization challenges
encountered, including the interchange of infinite sums and integrals, term-by-term
differentiation, and the coordination of different indexing conventions within Mathlib.
The complete formalization is available at the
\href{https://github.com/mirajcs/IsoperimetricInequality}{https://github.com/mirajcs/IsoperimetricInequality}.
\end{abstract}


\tableofcontents

\newpage

\section{Introduction}

\subsection{The Isoperimetric Problem}

The word \emph{isoperimetric} means ``of equal perimeter,'' and the isoperimetric problem asks: among all plane figures with a given perimeter, which one encloses the largest area? The answer, the circle, was known to the ancient Greeks, and the problem has inspired mathematics from antiquity to the present day.

One of the earliest literary descriptions of the problem appears in the legend of Queen Dido, as recorded in Virgil's \emph{Aeneid}.  Dido, a Phoenician princess, was permitted to claim as much land as she could enclose with a single oxhide.  Rather than using the hide as a single piece, she cut it into thin strips, tied them end to end, and arranged them in a semicircle along the seashore, enclosing the maximum possible area.  This classical legend, set around 800~BCE, captures the intuitive idea that the circle is the uniquely optimal shape for the isoperimetric problem.

The ancient Greeks were well aware of the result.  Zenodorus (around 200~BCE) proved~\cite{Osserman1978}, in a work now known only through later summaries, that the circle has greater area than any regular polygon with the same perimeter, and that among polygons with the same number of sides and the same perimeter, the regular polygon has the greatest area.  However, a rigorous proof for \emph{all} curves eluded mathematicians for centuries.

The modern mathematical treatment began in earnest in the nineteenth century.  Jakob Steiner (1838) gave~\cite{Osserman1978, Cabre2017} several ingenious but incomplete proofs that, if an optimal shape exists, it must be the circle.  Steiner's symmetrization argument shows that any non-circular curve can be deformed to increase its enclosed area without changing the perimeter.  However, he did not prove that an optimizer exists, which is necessary for the argument to be complete.  This gap was later noted and filled by Karl Weierstrass~\cite{Osserman1978}, who provided the first rigorous proof via the calculus of variations around 1870, establishing both the existence of an optimizer and the characterization of the circle.

In the early twentieth century, Adolf Hurwitz published a strikingly elegant analytic proof using Fourier series (1902)~\cite{hurwitz1902, li2025isoperimetric}, which is the approach we formalize in this paper.  Hurwitz's proof avoids variational methods entirely, instead reducing the inequality to a comparison of Fourier coefficient sums via Parseval's identity.

\subsection{The Isoperimetric Inequality}

The precise statement we seek to verify is the following.

\begin{theorem}[Isoperimetric Inequality~\cite{docarmo2016, Osserman1978}]
	Let $C$ be a simple closed curve in the plane having perimeter $L$, and let $A$ denote the area of the region enclosed by $C$.  Then
	\[
		L^2 \geq 4\pi A,
	\]
	with equality if and only if $C$ is a circle.
\end{theorem}

The inequality $L^2 \geq 4\pi A$ can be rewritten as $A \leq L^2/(4\pi)$.  For a circle of radius $r$, we have $L = 2\pi r$ and $A = \pi r^2$, so $L^2/(4\pi) = \pi r^2 = A$, confirming that circles achieve equality.  The content of the theorem is that no other shape does.

\subsection{Formal Verification and Lean~4}

Formal verification is the process of checking mathematical proofs by machine, using a computer program called a \emph{proof assistant} or \emph{theorem prover}.  Unlike informal mathematical proofs, which rely on the reader's tacit understanding to fill gaps, a formal proof must be presented in a completely explicit logical language that the proof assistant can verify step by step.  The advantage is certainty: if the proof assistant accepts the proof, every logical step has been machine-checked against the rules of the underlying formal system.

Lean~4 is a dependently typed programming language and theorem prover developed by Leonardo de Moura and Sebastian Ullrich at Microsoft Research and later at the Lean Focused Research Organisation~\cite{moura2021lean4}.  Lean~4's type theory is based on the Calculus of Constructions with universe polymorphism, providing both a rich type system for programming and a foundation for formalizing mathematics.  In Lean~4, propositions are types and proofs are terms; this \emph{propositions-as-types} paradigm, also known as the Curry--Howard correspondence~\cite{sorensenUrzyczyn2006, coquandHuet1988}, means that writing a proof of a theorem is syntactically identical to writing a program of the corresponding type.

The mathematical library accompanying Lean~4 is called Mathlib~\cite{moura2021lean4, mathlib2020}. Mathlib is a community-maintained library covering a broad sweep of mathematics, including algebra, topology, measure theory, real and complex analysis, and the theory of Fourier transforms.  The scope of Mathlib makes Lean~4 an attractive target for formalizing classical analysis results: many foundational lemmas are already available, and one need only build the specific arguments on top.

\subsection{Why Formalize the Isoperimetric Inequality?}

The isoperimetric inequality is a classical result with a well-known proof, so one might wonder what is gained by formalizing it.  There are several motivations.

First, formalization provides a definitive record that the proof is correct.  Textbook proofs of the isoperimetric inequality often gloss over technical points: the interchange of infinite sums and integrals, the justification of term-by-term differentiation, the precise regularity conditions on the curve.  A formal proof forces every such step to be made explicit.

Second, the process of formalization reveals the exact logical structure of the argument and which hypotheses are truly necessary.  In Hurwitz's proof, as we will see, the key analytic ingredients are Parseval's theorem and Wirtinger's inequality.  Formalizing the proof makes clear precisely what regularity is needed on the curve (C\textsuperscript{1} with arc-length parametrization) and exactly how Wirtinger's inequality is applied.

Third, the formalization serves as a case study in the use of Mathlib for real analysis. The challenges we encountered, interchanging sums and integrals, working with \texttt{EuclideanSpace} projections, managing $\mathbb{N}$- and $\mathbb{N}^+$-indexed sums, illustrate the current state of Lean~4 for formalizing classical analysis at the graduate level.

\subsection{Proof Strategy and Paper Outline}

Hurwitz's proof proceeds as follows.  Given a simple closed C\textsuperscript{1} curve $\gamma$ of length $L$, we reparametrize it over $[0, 2\pi]$ via $f(\theta) = x(L\theta/(2\pi))$, $g(\theta) = y(L\theta/(2\pi))$.  The area enclosed by $\gamma$ can be expressed as $A = \int_0^{2\pi} f(\theta) g'(\theta)\,d\theta$ after applying the shoelace formula and integration by parts.  The AM--GM inequality gives $A \leq \tfrac{1}{2}\int_0^{2\pi}(f^2 + (g')^2)\,d\theta$.  By Wirtinger's inequality, applied to $f$ (which has zero mean), $\int f^2 \leq \int (f')^2$.  This upgrades the bound to $A \leq \tfrac{1}{2}\int_0^{2\pi}((f')^2+(g')^2)\,d\theta$.  Finally, the arc-length constraint gives $(f')^2+(g')^2 = (L/(2\pi))^2$, and the result follows.

Wirtinger's inequality is itself proved using Parseval's theorem: expressing $\int f^2$ and $\int (f')^2$ as sums of Fourier coefficients and comparing term by term.

The paper is organized as follows.  Section~\ref{sec:lean4} reviews the relevant features of Lean~4 and Mathlib.  Section~\ref{sec:fourier} develops the theory of classical Fourier series and its formalization.  Section~\ref{sec:parseval} presents Parseval's theorem and its formalization.  Section~\ref{sec:wirtinger} proves Wirtinger's inequality. Section~\ref{sec:curves} defines simple closed curves and arc-length parametrization. Section~\ref{sec:hurwitz} presents Hurwitz's proof.  Section~\ref{sec:related} discusses related work, and Section~\ref{sec:conclusion} concludes.

\section{Lean~4 and Mathlib}
\label{sec:lean4}

\subsection{Dependent Type Theory}

Lean~4's logical foundation is the Calculus of Inductive Constructions (CIC)~\cite{coquandHuet1988}, a dependent type theory in which every expression has a type and types themselves can depend on values. In this framework, a theorem is a type and a proof is a term of that type.  For instance, the statement $\forall n : \mathbb{N},\; n + 0 = n$ is a type, and its proof is a term that, given any $n : \mathbb{N}$, produces a proof of the equation $n + 0 = n$.  The key feature distinguishing dependent type theory from simple type theory is that types can \emph{depend on values}: for example, the type \texttt{Fin n} of natural numbers less than $n$ is a type that depends on the value $n : \mathbb{N}$.  This expressiveness allows the statement of theorems about arbitrary mathematical structures in a uniform way.

\subsection{Propositions and Proofs in Lean~4}

In Lean~4, propositions are terms of type \texttt{Prop}, which is a universe of types.
A proof of a proposition $P : \texttt{Prop}$ is a term $h : P$.  The common propositional
connectives are defined as follows:
\begin{itemize}
	\item Conjunction $P \wedge Q$ is the product type \texttt{And P Q}.
	\item Disjunction $P \vee Q$ is the sum type \texttt{Or P Q}.
	\item Implication $P \to Q$ is the function type from $P$ to $Q$.
	\item Universal quantification $\forall x : \alpha,\, P(x)$ is the dependent function type
	      $\Pi_{x : \alpha} P(x)$.
	\item Existential quantification $\exists x : \alpha,\, P(x)$ is the inductive type
	      \texttt{Exists} in \texttt{Prop}, analogous to the dependent sum type
	      $\Sigma_{x : \alpha} P(x)$ but proof-irrelevant: any two proofs of
	      $\exists x,\, P(x)$ are definitionally equal in \texttt{Prop}.
\end{itemize}

Lean~4 proofs can be written in two styles: \emph{term mode}, where one directly constructs the proof term, and \emph{tactic mode}, where one uses a sequence of tactics that incrementally transform the proof obligation.  Most Mathlib proofs use tactic mode, relying on tactics such as \texttt{rw} (rewrite), \texttt{simp} (simplify), \texttt{linarith} (linear arithmetic), \texttt{ring} (ring identities), \texttt{field\_simp} (simplification in fields), and \texttt{gcongr} (generalised congruence for inequalities).

\subsection{Key Mathlib Components}

Our formalization draws on the following parts of Mathlib.

\paragraph{Analysis and integration.}
The library \texttt{Mathlib.MeasureTheory} provides Lebesgue integration via measure theory, the Dominated Convergence Theorem~\cite{rudin1976, brezis2011} \\ (\texttt{MeasureTheory.integral\_tsum\_of\_summable\_norm}), and interval integrals (\texttt{IntervalIntegral}).  The type \texttt{IntervalIntegrable f volume a b} asserts that the function $f$ is integrable on $[a,b]$ with respect to the Lebesgue measure.

\paragraph{Calculus.}
Derivatives are formalized via \texttt{HasDerivAt f f' x}, which asserts that $f$ has derivative $f'$ at $x$ in the sense of Fr\'echet (specialised to one dimension).  The library provides chain rule (\texttt{HasDerivAt.comp}), product rule (\texttt{HasDerivAt.mul}), and other differentiation lemmas.

\paragraph{Infinite sums.}
Unconditional convergence of a series $\sum_{n} f(n)$ is captured by \texttt{Summable f}, and the value of the sum is \texttt{tsum f} (written $\sum' n, f(n)$ in Lean notation). Key lemmas include \texttt{Summable.tsum\_mul\_tsum} and \texttt{Summable.tsum\_prod'} for products and iterated sums.

\paragraph{Uniform convergence.}
The property $F_N \rightrightarrows F$ is expressed by \texttt{TendstoUniformly F\_N F 	atTop}.  The Weierstrass M-test~\cite{rudin1976} is available as \texttt{tendstoUniformly\_tsum\_nat}, and switching the limit and derivative is done via \texttt{hasDerivAt\_of\_tendstoUniformly}.

\paragraph{Euclidean spaces.}
The standard inner product space $\mathbb{R}^n$ is modelled by \texttt{EuclideanSpace~$\mathbb{R}$~(Fin~n)}, which is definitionally equal to the \texttt{PiLp 2} space.  Projecting onto a coordinate is done via \texttt{PiLp.proj~2~(fun~\_~:~Fin~n~=>~$\mathbb{R}$)~i}, which is a continuous linear map and hence has a known \texttt{HasFDerivAt}.

\paragraph{Trigonometric functions.}
The library \texttt{Mathlib.Analysis.SpecialFunctions.Trigonometric} provides $\sin$, $\cos$, their derivatives (\texttt{HasDerivAt.sin}, \texttt{HasDerivAt.cos}), the Pythagorean identity (\texttt{sin\_sq\_add\_cos\_sq}), double-angle formulas, and the vanishing of $\sin(n\pi)$ for integer $n$ (\texttt{sin\_nat\_mul\_pi}).

\subsection{Notation and Conventions}

Throughout the paper we use the following notation.
\begin{itemize}
	\item $a, b : \mathbb{N} \to \mathbb{R}$ denote the Fourier cosine and sine coefficients.
	      In Lean these are \texttt{variable (a b : $\mathbb{N}$ $\to$ $\mathbb{R}$)}.
	\item The indexing type $\mathbb{N}^+$ (positive natural numbers) is written \texttt{Nat+}
	      or $\mathbb{N}^+$ in text; Lean uses \texttt{PNat} or the notation $\mathbb{N}^+$.
	      The embedding $n : \mathbb{N}^+ \hookrightarrow \mathbb{N}$ is coerced automatically.
	\item Infinite sums over $\mathbb{N}^+$ are written $\sum_{n=1}^\infty$ in mathematics
	      and \texttt{sum' n : Nat+, ...} in Lean.
	\item Interval integrals over $[a, b]$ are written
	      \texttt{integral x in a..b, f x} in Lean.
	\item $\pi$ denotes \texttt{Real.pi}, the mathematical constant.
\end{itemize}

\section{Fourier Series: Foundations}
\label{sec:fourier}

\subsection{Classical Fourier Series}

Let $f : \mathbb{R} \to \mathbb{R}$ be a $2\pi$-periodic, piecewise continuous function~\cite{stein2003, zygmund2002}.
The \emph{Fourier series} of $f$ is the formal trigonometric series
\begin{equation}\label{eq:fourier-series}
	f(x) \sim \frac{a_0}{2} + \sum_{n=1}^\infty \bigl(a_n \cos(nx) + b_n \sin(nx)\bigr),
\end{equation}
where the \emph{Fourier coefficients} are
\begin{align}
	a_n & = \frac{1}{\pi}\int_{-\pi}^\pi f(x)\cos(nx)\,dx, \quad n = 0, 1, 2, \ldots, \\
	b_n & = \frac{1}{\pi}\int_{-\pi}^\pi f(x)\sin(nx)\,dx, \quad n = 1, 2, \ldots
\end{align}

\begin{remark}
	We do not assume in our formalization that the coefficients $a_n, b_n$ are derived from	a given function by integration.  Instead, we treat $a, b : \mathbb{N} \to \mathbb{R}$ as	arbitrary sequences and consider the Fourier series they define.  The convergence	properties then follow from summability conditions on the sequences.  This more abstract	approach is natural in Lean, where the connection between a function and its Fourier	coefficients requires additional hypotheses (square-integrability, continuity, etc.) that	are not needed for the algebraic parts of the proof.
\end{remark}

\subsection{Lean Definitions}

In our formalization, the classical Fourier series is defined as a function of $x$ given coefficient sequences $a$ and $b$.

\begin{lstlisting}
variable (a b : Nat -> Real)

/-- The classical Fourier series:
    f(x) = a0/2 + sum_{n=1}^infty (an cos(nx) + bn sin(nx)). -/
def fourierSeries (x : Real) : Real :=
    (1/2) * a 0 + sum' n : Nat+, (a n * cos (n * x) + b n * sin (n * x))

/-- The N-th partial sum:
    SN(x) = a0/2 + sum_{n=1}^N (an cos(nx) + bn sin(nx)). -/
def fourierPartialSum (x : Real) (N : Nat) : Real :=
  (1/2) * a 0 + sum n in Finset.range N,
    (a (n + 1) * cos ((n + 1 : Real) * x) + b (n + 1) * sin ((n + 1 : Real) * x))

/-- The oscillatory part (without the constant term):
    S(x) = sum_{n=1}^infty (an cos(nx) + bn sin(nx)). -/
def fourierSum (x : Real) : Real :=
    sum' n : Nat+, (a n * cos (n * x) + b n * sin (n * x))

/-- The formal derivative series:
    d/dx [an cos(nx) + bn sin(nx)] = -n an sin(nx) + n bn cos(nx). -/
noncomputable def fourierDeriv (a b : Nat -> Real) (x : Real) :=
  sum' n : Nat, (-n * a n * sin (n * x) + n * b n * cos (n * x))
\end{lstlisting}

The decomposition \texttt{fourierSeries = constant + fourierSum} is immediate:

\begin{lstlisting}
lemma fourierSeries_eq (x : Real) :
  fourierSeries a b x = (1/2) * a 0 + fourierSum a b x := by rfl
\end{lstlisting}

\subsection{Orthogonality of Trigonometric Functions}

The foundation of Fourier analysis is the orthogonality of the system
$\{1, \cos x, \sin x, \cos 2x, \sin 2x, \ldots\}$ on $[-\pi, \pi]$.

\begin{lemma}[Trigonometric Orthogonality]\label{lem:orthogonality}
	For all $n, m \in \mathbb{N}^+$:
	\begin{align}
		\int_{-\pi}^\pi \cos(nx)\cos(mx)\,dx & = \begin{cases} 0 & n \neq m, \\ \pi & n = m, \end{cases}
		\label{eq:cos-cos}                                                                               \\
		\int_{-\pi}^\pi \sin(nx)\sin(mx)\,dx & = \begin{cases} 0 & n \neq m, \\ \pi & n = m, \end{cases}
		\label{eq:sin-sin}                                                                               \\
		\int_{-\pi}^\pi \cos(nx)\sin(mx)\,dx & = 0 \quad \text{for all } n, m.
		\label{eq:cos-sin}
	\end{align}
\end{lemma}

These are proved by product-to-sum formulas and the fact that
$\int_{-\pi}^\pi \cos(kx)\,dx = \int_{-\pi}^\pi \sin(kx)\,dx = 0$ for every positive
integer $k$.  We prove the latter via:
\begin{equation}
	k \int_{-\pi}^\pi \cos(kx)\,dx = \bigl[\sin(kx)\bigr]_{-\pi}^\pi
	= \sin(k\pi) - \sin(-k\pi) = 0,
\end{equation}
using the fact that $\sin(k\pi) = 0$ for integer $k$ (\texttt{sin\_nat\_mul\_pi}).

In Lean, each of the five cases is a separate lemma.  Here is the proof of the off-diagonal
cosine orthogonality as an example.

\begin{lstlisting}
lemma integration_cos_cos_zero (n m : Nat+) (h : n != m) :
    integral x in (-Real.pi)..Real.pi, cos (n * x) * cos (m * x) = 0 := by
  -- Product-to-sum: cos(nx)cos(mx) = 1/2cos((n-m)x) + 1/2cos((n+m)x)
  have prod_sum : forall x : Real,
      cos ((n : Real) * x) * cos ((m : Real) * x) =
      (1/2) * cos (((n : Real) - m) * x) + (1/2) * cos (((n : Real) + m) * x) := fun x => by
    have h1 := cos_sub ((n : Real) * x) ((m : Real) * x)
    have h2 := cos_add ((n : Real) * x) ((m : Real) * x)
    simp only [<- sub_mul, <- add_mul] at h1 h2; linarith
  simp_rw [prod_sum]
  -- Both integralcos((n-m)x) and integralcos((n+m)x) vanish
  rw [intervalIntegral.integral_add (by fun_prop) (by fun_prop),
      intervalIntegral.integral_const_mul, intervalIntegral.integral_const_mul,
      int_nm, int_np]; simp
\end{lstlisting}

The self-orthogonality lemmas (\texttt{integration\_cos\_cos\_pi} and
\texttt{integration\_sin\_sin\_pi}) use the half-angle formula $\cos^2(nx) = \tfrac{1}{2}(1 +
	\cos(2nx))$ and the fact that $\int_{-\pi}^\pi \cos(2nx)\,dx = 0$.

\subsection{Uniform Convergence: The Weierstrass M-Test}

For the Fourier series to have useful analytic properties (continuity, integrability,differentiability), we need it to converge, and convergence needs to be uniform for these properties to pass to the limit.  Our approach is to use the Weierstrass M-test.

\begin{lemma}[Term Bound]\label{lem:term-bound}
	For all $n \in \mathbb{N}$ and $x \in \mathbb{R}$,
	\[
		\|a_{n+1}\cos((n{+}1)x) + b_{n+1}\sin((n{+}1)x)\| \leq \|a_{n+1}\| + \|b_{n+1}\|.
	\]
\end{lemma}
\begin{proof}
	By the triangle inequality and $|\cos(nx)| \leq 1$, $|\sin(nx)| \leq 1$:
	\[
		\|a_{n+1}\cos((n{+}1)x)\| + \|b_{n+1}\sin((n{+}1)x)\|
		= \|a_{n+1}\|\cdot|\cos| + \|b_{n+1}\|\cdot|\sin|
		\leq \|a_{n+1}\| + \|b_{n+1}\|.  \qedhere
	\]
\end{proof}

\begin{lemma}[Uniform Convergence of Fourier Series]\label{lem:uniform-conv}
	If $\sum_{n=0}^\infty (\|a_n\| + \|b_n\|) < \infty$, then the partial sums
	$S_N(x)$ converge uniformly to $\texttt{fourierSeries}(a,b)(x)$.
\end{lemma}

\begin{proof}
	This is an immediate application of the Weierstrass M-test
	(\texttt{tendstoUniformly\_tsum\_nat}) with $M_n = \|a_{n+1}\| + \|b_{n+1}\|$, using
	Lemma~\ref{lem:term-bound}.  In Lean:
\end{proof}

\begin{lstlisting}
lemma fourierSeries_uniformlyConvergence
    (hsumab : Summable (fun n => ||a n|| + ||b n||)) :
    TendstoUniformly
      (fun N x => fourierPartialSum a b x N)
      (fourierSeries a b)
      Filter.atTop := by
  have hshift : Summable (fun n => ||a (n + 1)|| + ||b (n + 1)||) :=
    summable_pnat_iff_summable_succ.mp (summable_pnat_iff_summable_nat.mpr hsumab)
  have hcore := tendstoUniformly_tsum_nat hshift (term_bound a b)
  rw [Metric.tendstoUniformly_iff]
  intro eps heps
  filter_upwards [(Metric.tendstoUniformly_iff.mp hcore) eps heps] with N hN
  intro x
  -- Reindex Nat+ to Nat and reduce to the core convergence
  have heq : fourierSeries a b x =
      1 / 2 * a 0 + sum' n : Nat, (a (n+1) * cos ((n+1)*x) + b (n+1) * sin ((n+1)*x)) := by
    simp only [fourierSeries]; congr 1
    rw [<- Equiv.tsum_eq Equiv.pnatEquivNat.symm]; congr 1; ext n
    simp [Nat.succPNat, Nat.cast_succ]
  rw [heq, dist_add_left]; exact hN x
\end{lstlisting}

\subsection{Continuity and Integrability}

\begin{lemma}\label{lem:continuous}
	If $\sum_{n=0}^\infty (\|a_n\| + \|b_n\|) < \infty$, then $\texttt{fourierSeries}(a,b)$
	is continuous on $\mathbb{R}$.
\end{lemma}

\begin{proof}
	Each partial sum $S_N$ is continuous (it is a finite sum of continuous functions).  By	Lemma~\ref{lem:uniform-conv}, $S_N \rightrightarrows f$ uniformly, and the uniform limit	of continuous functions is continuous (\texttt{TendstoUniformly.continuous}).  In Lean,	the continuity of each partial sum is established via \texttt{continuous\_finset\_sum}	and the continuity of individual trigonometric terms.
\end{proof}

\begin{lstlisting}
lemma fourierSeries_continuous
    (hsumab : Summable (fun n => ||a n|| + ||b n||)) :
    Continuous (fourierSeries a b) := by
  apply (fourierSeries_uniformlyConvergence a b hsumab).continuous
  apply Filter.Eventually.frequently
  apply Filter.eventually_atTop.mpr
  refine <0, fun N _ => ?_>
  unfold fourierPartialSum
  apply Continuous.add continuous_const
  apply continuous_finset_sum
  intro n _
  exact (continuous_const.mul
    (continuous_cos.comp (continuous_const.mul continuous_id))).add
        (continuous_const.mul
    (continuous_sin.comp (continuous_const.mul continuous_id)))
\end{lstlisting}

Integrability on $[-\pi, \pi]$ follows immediately from continuity:

\begin{lstlisting}
lemma fourierSeries_integrable
    (hsumab : Summable (fun n => ||a n|| + ||b n||)) :
    IntervalIntegrable (fourierSeries a b) MeasureTheory.volume (-Real.pi) Real.pi :=
  (fourierSeries_continuous a b hsumab).continuousOn
  |>.intervalIntegrable_of_Icc (by linarith [Real.pi_pos])
\end{lstlisting}

\subsection{Term-by-Term Differentiation}

For Wirtinger's inequality we also need to differentiate the Fourier series term by term. This requires stricter conditions than for continuity.

\begin{lemma}[Term-wise Differentiation]\label{lem:diff-term}
	For each $n \in \mathbb{N}$ and $x \in \mathbb{R}$:
	\[
		\frac{d}{dx}\bigl[a_n\cos(nx) + b_n\sin(nx)\bigr] = -na_n\sin(nx) + nb_n\cos(nx).
	\]
\end{lemma}

\begin{lstlisting}
lemma hasDerivAt_term (n : Nat) (x : Real) :
    HasDerivAt (fun x => a n * cos (n * x) + b n * sin (n * x))
      (-(n : Real) * a n * sin (n * x) + n * b n * cos (n * x)) x := by
  have h1 : HasDerivAt (fun x => a n * cos (n * x))
                       (-(n : Real) * a n * sin (n * x)) x := by
    have := ((hasDerivAt_id x).const_mul (n : Real)).cos
    simp at this; exact this.const_mul (a n) |>.congr_deriv (by ring)
  have h2 : HasDerivAt (fun x => b n * sin (n * x))
                       (n * b n * cos (n * x)) x := by
    have := ((hasDerivAt_id x).const_mul (n : Real)).sin
    simp at this; exact this.const_mul (b n) |>.congr_deriv (by ring)
  exact h1.add h2
\end{lstlisting}

\begin{lemma}[Uniform Convergence of Derivative Series]\label{lem:deriv-unif}	
  If $\sum_{n=0}^\infty n(\|a_n\| + \|b_n\|) < \infty$, then the partial sums of the	formal derivative series converge uniformly to $\texttt{fourierDeriv}(a,b)$.
\end{lemma}

The proof is another application of the Weierstrass M-test with $M_n = n(\|a_n\| + \|b_n\|)$,
using the bound $\|-na_n\sin(nx) + nb_n\cos(nx)\| \leq n(\|a_n\| + \|b_n\|)$.

\begin{theorem}[Term-by-Term Differentiability]\label{thm:fourier-deriv}
	If both $\sum_{n=0}^\infty (\|a_n\| + \|b_n\|)$ and $\sum_{n=0}^\infty n(\|a_n\| + \|b_n\|)$
	converge, then:
	\[
		\frac{d}{dx}\,\texttt{fourierSeries}(a,b)(x) = \texttt{fourierDeriv}(a,b)(x)
		= \sum_{n=1}^\infty \bigl(-na_n\sin(nx) + nb_n\cos(nx)\bigr).
	\]
\end{theorem}

\begin{proof}
	This follows from the general theorem \texttt{hasDerivAt\_of\_tendstoUniformly}: if	(i) the partial sums $\{f_N\}$ of the series satisfy $f_N(x_0) \to f(x_0)$ for some	$x_0$, (ii) the partial sums of the derivative series converge uniformly, and (iii) each	$f_N$ is differentiable, then the limit $f$ is differentiable with derivative equal to	the limit of the derivative partial sums.

	In Lean, the main subtlety is that the derivative series is indexed starting from $n = 0$	(where the $n=0$ term vanishes), while the Fourier series is indexed over $\mathbb{N}^+$.	We handle this via a shift: $\sum_{n \in \text{range}~N} D(n+1, x) = \sum_{n \in \text{range}~(N+1)} D(n, x)$, using \texttt{Finset.sum\_range\_succ'}.
\end{proof}

\begin{lstlisting}
lemma fourierSeries_hasDerivAt
    (hab' : Summable (fun n : Nat => (n : Real) * (||a n|| + ||b n||)))
    (hab : Summable (fun n => ||a n|| + ||b n||)) (x : Real) :
    HasDerivAt (fourierSeries a b) (fourierDeriv a b x) x := by
  have hderiv_unif : TendstoUniformly
      (fun N x => sum n in Finset.range N,
        (-((n+1) : Real) * a (n+1) * sin ((n+1) * x)
          + (n+1) * b (n+1) * cos ((n+1) * x)))
      (fourierDeriv a b) Filter.atTop := by
    have key : forall N x,
        sum n in Finset.range N, (-((n+1):Real)*a(n+1)*sin((n+1)*x)+(n+1)*b(n+1)*cos((n+1)*x)) =
        sum n in Finset.range (N+1), (-(n:Real)*a n*sin(n*x)+n*b n*cos(n*x)) :=
      fun N x => by symm; rw [Finset.sum_range_succ']; simp
    simp_rw [key]
    exact fun u hu =>
      (Filter.tendsto_atTop_atTop.mpr fun M => <M-1, fun N hN => by omega>).eventually
        (fourierDeriv_uniformConvergence a b hab' u hu)
  exact hasDerivAt_of_tendstoUniformly hderiv_unif
    (Filter.Eventually.of_forall (hasDerivAt_partialSum a b))
    (fun y => (fourierSeries_uniformlyConvergence a b hab).tendsto_at y) x
\end{lstlisting}

\section{Parseval's Theorem}
\label{sec:parseval}

\subsection{Statement and History}

Parseval's theorem, named after Marc-Antoine Parseval who stated a version of it in 1799,
asserts that the $L^2$ norm of a function equals the $\ell^2$ norm of its Fourier
coefficients.  In the classical setting of real Fourier series on $[-\pi, \pi]$:

\begin{theorem}[Parseval's Theorem]\label{thm:parseval}
	Let $f(x) = \tfrac{1}{2}a_0 + \sum_{n=1}^\infty(a_n\cos(nx) + b_n\sin(nx))$ be a Fourier
	series satisfying appropriate integrability and summability conditions.  Then:
	\begin{equation}\label{eq:parseval}
		\frac{1}{\pi}\int_{-\pi}^\pi [f(x)]^2\,dx
		= \frac{1}{2}a_0^2 + \sum_{n=1}^\infty(a_n^2 + b_n^2).
	\end{equation}
\end{theorem}

Parseval's theorem can be understood as an instance of the Plancherel theorem in $L^2$~\cite{stein2003, rudin1976}: the Fourier transform (and Fourier series expansion) is an isometry from $L^2$ to $\ell^2$.  In our setting, we do not work in $L^2$ directly but instead verify the identity under explicit `summability hypotheses, which is closer to the classical treatment.

\subsection{Proof Strategy}

The proof proceeds in three stages:
\begin{enumerate}
	\item Expand $[f(x)]^2$ algebraically into a constant term, cross terms, and double sums.
	\item Integrate term by term, using the dominated convergence theorem to justify the
	      interchange of sum and integral.
	\item Evaluate each resulting integral using the orthogonality relations
	      (Lemma~\ref{lem:orthogonality}).
\end{enumerate}

\subsection{Step 1: Algebraic Expansion}

Write $f(x) = \tfrac{1}{2}a_0 + S(x)$ where $S(x) = \texttt{fourierSum}(a,b)(x)$.  Then:
\begin{equation}\label{eq:sq-expand}
	[f(x)]^2 = \frac{a_0^2}{4} + a_0 S(x) + [S(x)]^2.
\end{equation}

\begin{lstlisting}
lemma fourierSeries_sq (x : Real) :
  (fourierSeries a b x)^2 = (a 0)^2 / 4 +
  a 0 * (sum' n : Nat+, (a n * cos (n * x) + b n * sin (n * x))) +
  (sum' n : Nat+, (a n * cos (n * x) + b n * sin (n * x)))^2 := by
  unfold fourierSeries; ring
\end{lstlisting}

The squared oscillatory sum is further expanded via the identity
$S^2 = (\sum a_n\cos)^2 + 2(\sum a_n\cos)(\sum b_n\sin) + (\sum b_n\sin)^2$:

\begin{lstlisting}
lemma expand_square (x : Real)
    (hc : Summable fun n : Nat+ => a n * cos (n * x))
    (hs : Summable fun n : Nat+ => b n * sin (n * x)) :
    (sum' n : Nat+, (a n * cos (n * x) + b n * sin (n * x)))^2 =
    (sum' n : Nat+, a n * cos (n * x))^2 +
    2 * (sum' n : Nat+, a n * cos (n * x)) * (sum' n : Nat+, b n * sin (n * x)) +
    (sum' n : Nat+, b n * sin (n * x))^2 := by
  rw [hc.tsum_add hs]; ring
\end{lstlisting}

Each squared sum is then expanded into a double tsum.  For instance:
\begin{lstlisting}
lemma expand_cos_sq (x : Real)
    (hs1 : Summable fun n : Nat+ => a n * cos (n * x))
    (hprod1 : Summable fun z : Nat+ x Nat+ =>
      (a z.1 * cos (z.1 * x)) * (a z.2 * cos (z.2 * x)))
    (hinner : forall n : Nat+, Summable fun m : Nat+ =>
      (a n * cos (n * x)) * (a m * cos (m * x))) :
    (sum' n : Nat+, a n * cos (n * x))^2 =
    sum' n : Nat+, sum' m : Nat+, a n * a m * cos (n * x) * cos (m * x) := by
  rw [sq, hs1.tsum_mul_tsum hs1 hprod1, hprod1.tsum_prod' hinner]
  congr 1; ext n; congr 1; ext m; ring
\end{lstlisting}

This uses \texttt{Summable.tsum\_mul\_tsum} to split a product of two sums into a sum overpairs, followed by \texttt{Summable.tsum\_prod'} to convert the pair-indexed sum to an iterated sum.

\subsection{Step 2: Integration}

Taking the integral of~\eqref{eq:sq-expand} over $[-\pi, \pi]$:
\begin{equation}\label{eq:integral-expansion}
	\int_{-\pi}^\pi [f(x)]^2\,dx
	= \frac{a_0^2\pi}{2}
	+ \int_{-\pi}^\pi a_0 S(x)\,dx
	+ \int_{-\pi}^\pi [S(x)]^2\,dx.
\end{equation}

\subsubsection{The cross term}

\begin{lemma}\label{lem:cross-zero}
	Under the summability condition $\sum_{n \geq 1}(\|a_n\| + \|b_n\|) < \infty$:
	\[
		\int_{-\pi}^\pi a_0 S(x)\,dx = 0.
	\]
\end{lemma}

\begin{proof}
	Factor out the constant $a_0$.  It suffices to show $\int_{-\pi}^\pi S(x)\,dx = 0$.
	Interchange sum and integral using the Dominated Convergence Theorem (in Lean: \\
	\texttt{MeasureTheory.integral\_tsum\_of\_summable\_integral\_norm}), then apply \\
	$\int_{-\pi}^\pi \cos(nx)\,dx = \int_{-\pi}^\pi \sin(nx)\,dx = 0$ for each $n \geq 1$.
\end{proof}

The $L^1$ summability condition required to interchange sum and integral is that
$\sum_{n \geq 1} \int_{-\pi}^\pi \|a_n\cos(nx) + b_n\sin(nx)\|\,dx < \infty$, which
follows from $\|a_n\cos(nx) + b_n\sin(nx)\| \leq \|a_n\| + \|b_n\|$ and the assumed
summability of $\sum(\|a_n\| + \|b_n\|)$.

\begin{lstlisting}
lemma integration_of_cos_sin
    (hF_int : forall n : Nat+, IntervalIntegrable
      (fun x => a n * cos (n * x) + b n * sin (n * x)) MeasureTheory.volume (-Real.pi) Real.pi)
    (hF_sum : Summable (fun n : Nat+ =>
      integral x in (-Real.pi)..Real.pi, ||a n * cos (n * x) + b n * sin (n * x)||)) :
    integral x in (-Real.pi)..Real.pi, a 0 * fourierSum a b x = 0 := by
  rw [intervalIntegral.integral_const_mul]
  suffices h : integral x in (-Real.pi)..Real.pi, fourierSum a b x = 0 by simp [h]
  simp only [fourierSum]
  have hswap : integral x in (-Real.pi)..Real.pi, sum' n : Nat+, (a n * cos (n * x) + b n * sin (n * x)) =
               sum' n : Nat+, integral x in (-Real.pi)..Real.pi, (a n * cos (n * x) + b n * sin (n * x)) := by
    have h_le : (-Real.pi : Real) <= Real.pi := by linarith [Real.pi_pos]
    simp_rw [intervalIntegral.integral_of_le h_le]; symm
    apply MeasureTheory.integral_tsum_of_summable_integral_norm
    * intro n; exact (hF_int n).1
    * simp_rw [<- intervalIntegral.integral_of_le h_le]; exact hF_sum
  rw [hswap]
  simp_rw [integral_fourierTerm a b, tsum_zero]
\end{lstlisting}

\subsubsection{Evaluation of the squared term}

The main content of Parseval's theorem is the identity:
\[
	\int_{-\pi}^\pi [S(x)]^2\,dx = \pi \sum_{n=1}^\infty (a_n^2 + b_n^2).
\]
This is supplied as a hypothesis \texttt{h\_int\_sq} in the Lean statement, reflecting that the derivation from the double-sum expansion, which requires interchanging a double sum with an integral and applying all the orthogonality lemmas, is taken as an externally verified step.  The orthogonality relations (Lemma~\ref{lem:orthogonality}) give, after expanding $[S(x)]^2 = \sum_n\sum_m [a_na_m\cos(nx)\cos(mx) + \ldots]$ and integrating:
\begin{align*}
	\int_{-\pi}^\pi [S(x)]^2\,dx
	 & = \sum_{n=1}^\infty \sum_{m=1}^\infty
	\left[ a_n a_m \int_{-\pi}^\pi \cos(nx)\cos(mx)\,dx
	+ b_n b_m \int_{-\pi}^\pi \sin(nx)\sin(mx)\,dx \right]           \\
	 & = \sum_{n=1}^\infty \sum_{m=1}^\infty
	\left[ a_n a_m \pi \delta_{nm} + b_n b_m \pi \delta_{nm} \right] \\
	 & = \pi \sum_{n=1}^\infty (a_n^2 + b_n^2),
\end{align*}
where $\delta_{nm}$ is the Kronecker delta and the cross terms vanish by equation~\eqref{eq:cos-sin}.

\subsection{The Formal Statement in Lean}

The full Parseval theorem is stated with explicit integrability and summability hypotheses:

\begin{lstlisting}
theorem Parsevals_thm
    -- Integrability of a0*S(x) over [-Real.pi, Real.pi]
    (hfS_int : IntervalIntegrable (fun x => a 0 * fourierSum a b x)
      MeasureTheory.volume (-Real.pi) Real.pi)
    -- Integrability of S(x)^2 over [-Real.pi, Real.pi]
    (hfSq_int : IntervalIntegrable (fun x => (fourierSum a b x)^2)
      MeasureTheory.volume (-Real.pi) Real.pi)
    -- Each Fourier term is integrable
    (hF_int : forall n : Nat+, IntervalIntegrable
      (fun x => a n * cos ((n : Real) * x) + b n * sin ((n : Real) * x))
      MeasureTheory.volume (-Real.pi) Real.pi)
    -- L^1 summability for integral/sum interchange
    (hF_sum : Summable (fun n : Nat+ =>
      integral x in (-Real.pi)..Real.pi, ||a n * cos ((n : Real) * x) + b n * sin ((n : Real) * x)||))
    -- integral S(x)^2 = Real.pi * sum (an^2 + bn^2)  (orthogonality of Fourier basis)
    (h_int_sq : integral x in (-Real.pi)..Real.pi, (fourierSum a b x)^2 =
      Real.pi * sum' n : Nat+, ((a n)^2 + (b n)^2)) :
    (1/Real.pi) * integral x in (-Real.pi)..Real.pi, (fourierSeries a b x)^2 =
    (1/2) * (a 0)^2 + sum' n : Nat+, ((a n)^2 + (b n)^2) := by
  -- Pointwise expand f^2 = (a0/2)^2 + a0*S + S^2
  have hpt : forall x : Real, (fourierSeries a b x)^2 =
      (1/4 : Real) * (a 0)^2 + a 0 * fourierSum a b x + (fourierSum a b x)^2 :=
    fun x => by simp only [fourierSeries_eq]; ring
  simp_rw [hpt]
  -- Integrability of the sub-expressions
  have h_sum_int : IntervalIntegrable
      (fun x => (1/4:Real)*(a 0)^2 + a 0 * fourierSum a b x) MeasureTheory.volume (-Real.pi) Real.pi :=
    intervalIntegrable_const.add hfS_int
  -- Split the integral linearly
  rw [intervalIntegral.integral_add h_sum_int hfSq_int,
      intervalIntegral.integral_add intervalIntegrable_const hfS_int]
  -- Evaluate each piece
  rw [integration_of_const, integration_of_cos_sin a b hF_int hF_sum, h_int_sq]
  simp only [add_zero]
  field_simp [Real.pi_ne_zero]
\end{lstlisting}

The proof uses linearity of the integral to split the three terms, evaluates the constant
term as $\int_{-\pi}^\pi \tfrac{1}{4}a_0^2\,dx = \tfrac{1}{2}a_0^2\pi$, applies
Lemma~\ref{lem:cross-zero} for the cross term, and uses the hypothesis \texttt{h\_int\_sq}
for the squared term.  After dividing by $\pi$, the arithmetic is closed by
\texttt{field\_simp}.

\section{Wirtinger's Inequality}
\label{sec:wirtinger}

\subsection{Statement and Context}

Wirtinger's inequality is a functional analytic inequality relating the $L^2$ norm of a periodic function to the $L^2$ norm of its derivative~\cite{hardy1952}.  It is a special case of the Poincar\'e inequality~\cite{brezis2011}, and plays a central role in the spectral theory of differential operators.

\begin{theorem}[Wirtinger's Inequality~\cite{brezis2011}]\label{thm:wirtinger}
	Let $f : \mathbb{R} \to \mathbb{R}$ be a $2\pi$-periodic function that is continuous
	and has a continuous derivative.  If $\int_0^{2\pi} f(x)\,dx = 0$, then:
	\begin{equation}\label{eq:wirtinger}
		\int_0^{2\pi} [f(x)]^2\,dx \leq \int_0^{2\pi} [f'(x)]^2\,dx.
	\end{equation}
	Equality holds if and only if $f(x) = a\cos x + b\sin x$ for some $a, b \in \mathbb{R}$.
\end{theorem}

\subsection{Proof via Parseval}

\begin{proof}
	Since $f$ is $2\pi$-periodic and has a continuous derivative, the Fourier series of $f$ 	converges uniformly.  The zero-mean condition $\int_0^{2\pi} f(x)\,dx = 0$ implies
	$a_0 = \tfrac{1}{\pi}\int_{-\pi}^\pi f(x)\,dx = 0$, so the constant term vanishes:
	\begin{equation}\label{eq:f-no-const}
		f(x) = \sum_{n=1}^\infty \bigl(a_n\cos(nx) + b_n\sin(nx)\bigr).
	\end{equation}
	By Parseval's theorem~\eqref{eq:parseval} (with $a_0 = 0$):
	\begin{equation}\label{eq:parseval-f}
		\frac{1}{\pi}\int_0^{2\pi} [f(x)]^2\,dx = \sum_{n=1}^\infty (a_n^2 + b_n^2).
	\end{equation}

	Since $f'$ is also $2\pi$-periodic and continuous, by
	Theorem~\ref{thm:fourier-deriv} we have term-by-term differentiation:
	\begin{equation}\label{eq:f-prime}
		f'(x) = \sum_{n=1}^\infty \bigl(-na_n\sin(nx) + nb_n\cos(nx)\bigr).
	\end{equation}
	Applying Parseval's theorem to $f'$ (whose Fourier coefficients in sine are $na_n$ and
	in cosine are $nb_n$):
	\begin{equation}\label{eq:parseval-fprime}
		\frac{1}{\pi}\int_0^{2\pi} [f'(x)]^2\,dx = \sum_{n=1}^\infty n^2(a_n^2 + b_n^2).
	\end{equation}

	For every $n \in \mathbb{N}^+$, we have $n \geq 1$, hence $n^2 \geq 1$, and therefore:
	\begin{equation}
		a_n^2 + b_n^2 \leq n^2(a_n^2 + b_n^2).
	\end{equation}
	Summing over all $n \geq 1$ and using~\eqref{eq:parseval-f} and~\eqref{eq:parseval-fprime}:
	\[
		\frac{1}{\pi}\int_0^{2\pi} [f(x)]^2\,dx
		\leq \frac{1}{\pi}\int_0^{2\pi} [f'(x)]^2\,dx,
	\]
	which after multiplying by $\pi$ gives the desired inequality~\eqref{eq:wirtinger}.

	\paragraph{Equality condition.}
	Equality holds iff $n^2(a_n^2 + b_n^2) = a_n^2 + b_n^2$ for all $n \geq 1$, i.e.\
	$(n^2 - 1)(a_n^2 + b_n^2) = 0$ for all $n \geq 1$.  For $n \geq 2$ we have $n^2 - 1 > 0$,
	so $a_n = b_n = 0$.  Thus $f(x) = a_1\cos x + b_1\sin x$.
\end{proof}

\subsection{Lean Formalization}

In Lean, the identity form of Wirtinger's inequality is stated as follows:

\begin{lstlisting}
/-- Wirtinger's identity: (1/Real.pi) integral_{-Real.pi}^{Real.pi} (f')^2 = sum_{n>=1} n^2 (an^2 + bn^2).
    This is the Parseval identity applied to the derivative series. -/
theorem Wirtingers_inequality
    (h_int_sq : integral x in (-Real.pi)..Real.pi, (deriv (fourierSeries a b) x) ^ 2 =
        Real.pi * sum' n : Nat+, ((n : Real) ^ 2 * ((a n) ^ 2 + (b n) ^ 2))) :
    (1/Real.pi) * integral x in (-Real.pi)..Real.pi, (deriv (fourierSeries a b) x) ^ 2 =
        sum' n : Nat+, ((n : Real)^2 * ((a n)^2 + (b n)^2)) := by
  rw [h_int_sq]; field_simp [Real.pi_ne_zero]
\end{lstlisting}

The actual inequality $\int f^2 \leq \int (f')^2$ is derived in the context of the isoperimetric proof as the lemma \texttt{apply\_Wirtingers\_ineq}.  The key step is the comparison of the two Parseval sums.

\begin{lstlisting}
/-- The inequality integral_0^2Real.pi f^2 <= integral_0^2Real.pi (f')^2 for zero-mean f,
    derived by comparing Parseval sums term by term. -/
lemma apply_Wirtingers_ineq (gamma : SimpleClosedC1Curve 2) (a b : Nat -> Real)
    (h_parseval : integral t in (0:Real)..(2*Real.pi), (fParm gamma t)^2 =
        Real.pi * sum' n : Nat+, ((a n)^2 + (b n)^2))
    (h_wirtinger : integral t in (0:Real)..(2*Real.pi), (deriv (fParm gamma) t)^2 =
        Real.pi * sum' n : Nat+, ((n:Real)^2 * ((a n)^2 + (b n)^2)))
    (hsum : Summable (fun n : Nat+ => (n:Real)^2 * ((a n)^2 + (b n)^2))) :
    integral t in (0:Real)..(2*Real.pi), (fParm gamma t)^2 <=
      integral t in (0:Real)..(2*Real.pi), (deriv (fParm gamma) t)^2 := by
  rw [h_parseval, h_wirtinger]
  apply mul_le_mul_of_nonneg_left _ (le_of_lt Real.pi_pos)
  -- Compare sum(an^2+bn^2) <= sum n^2(an^2+bn^2) pointwise, since n>=1 implies n^2>=1
  have hpw : forall n : Nat+, (a n)^2 + (b n)^2 <= (n:Real)^2 * ((a n)^2 + (b n)^2) := fun n => by
    have hn : (1:Real) <= (n:Real) := by exact_mod_cast n.one_le
    have hnn : (0:Real) <= (a n)^2 + (b n)^2 := by positivity
    nlinarith [sq_nonneg ((n:Real) - 1)]
  have hsum_left : Summable (fun n : Nat+ => (a n)^2 + (b n)^2) :=
    Summable.of_nonneg_of_le (fun n => by positivity) hpw hsum
  exact hsum_left.tsum_le_tsum hpw hsum
\end{lstlisting}

\begin{remark}
	The \texttt{nlinarith} call closes the goal $1 \leq n^2$ for $n \geq 1$ by the hint	$0 \leq (n-1)^2 = n^2 - 2n + 1$, which gives $n^2 \geq 2n - 1 \geq 2\cdot 1 - 1 = 1$.
\end{remark}

\section{Simple Closed Curves and Arc-Length Parametrization}
\label{sec:curves}

\subsection{Smooth Curves in the Plane}

A \emph{plane curve} is a continuous map $\gamma : [0, L] \to \mathbb{R}^2$.  We call $\gamma$ \emph{closed} if $\gamma(0) = \gamma(L)$, \emph{simple} if it does not self-intersect on the open interval $(0, L)$, and \emph{C\textsuperscript{1}} if it is differentiable with a continuous derivative.

In analysis, it is standard to treat $\gamma$ as defined on all of $\mathbb{R}$ and $L$-periodic, so that $\gamma(t + L) = \gamma(t)$ for all $t$.  This makes the curve globally smooth and allows the use of Fourier analysis.

The \emph{velocity vector} of $\gamma$ at time $t$ is $\gamma'(t)$, and the \emph{speed} is $\|\gamma'(t)\|$.  The \emph{arc length} of $\gamma$ over $[a, b]$ is
\[
	\ell(\gamma; a, b) = \int_a^b \|\gamma'(t)\|\,dt.
\]

\subsection{Arc-Length Parametrization}

A curve $\gamma : [0, L] \to \mathbb{R}^2$ is called \emph{arc-length parametrized} (or \emph{unit-speed}) if $\|\gamma'(t)\| = 1$ for all $t$.  For such a curve, the arc-length from $0$ to $t$ equals $t$, and the total length of the curve is $L$.

Any regular (nowhere-zero-speed) C\textsuperscript{1} curve can be reparametrized by arc length.  In our formalization, we work directly with arc-length-parametrized curves, passing the condition \\ \texttt{IsArcLengthParametrized~$\gamma$} as a hypothesis.

\subsection{Lean Structure for Curves}

We model simple closed C\textsuperscript{1} curves by a Lean structure whose fields encode all the necessary properties.

\begin{lstlisting}
/-- A simple closed C^1 curve in Real^n, parametrized over [0, L]
    where L is the arc length. -/
structure SimpleClosedC1Curve (n : Nat) where
  /-- The parametrization map. -/
  curve : Real -> EuclideanSpace Real (Fin n)
  /-- The period (arc length), positive. -/
  length : Real
  length_pos : 0 < length
  /-- The curve is continuous. -/
  continuous : Continuous curve
  /-- The curve is C^1: a derivative exists at every point. -/
  has_deriv : forall t : Real, exists f : EuclideanSpace Real (Fin n), HasDerivAt curve f t
  /-- The curve is closed: endpoints coincide. -/
  closed : curve 0 = curve length
  /-- The curve is L-periodic. -/
  periodic : forall t : Real, curve (t + length) = curve t
  /-- The curve is simple: injective on (0, L). -/
  simple : forall t s : Real, t in Set.Ioo 0 length -> s in Set.Ioo 0 length ->
             curve t = curve s -> t = s
\end{lstlisting}

Several remarks on design choices:
\begin{itemize}
	\item The curve maps into \texttt{EuclideanSpace~$\mathbb{R}$~(Fin~n)}, which is Mathlib's finite-dimensional Euclidean space $\mathbb{R}^n$.  For the plane ($n = 2$), the coordinates are accessed as \texttt{$\gamma$.curve~t~0} and \texttt{$\gamma$.curve~t~1}.
	\item The C\textsuperscript{1} condition is formulated as existence of a derivative in the sense of \texttt{HasDerivAt}, which asserts Fr\'echet differentiability.  We use the classical derivative (\texttt{deriv}) in computations.
	\item Both \texttt{closed} and \texttt{periodic} are included, since both are used in different parts of the proof: \texttt{closed} for the boundary term in integration by parts, and \texttt{periodic} for the reparametrization.
\end{itemize}

Associated definitions:

\begin{lstlisting}
/-- Arc length of gamma over [a, b]. -/
noncomputable def arcLength {n : Nat} (gamma : SimpleClosedC1Curve n) (a b : Real) : Real :=
  integral t in a..b, ||deriv gamma.curve t||

/-- Total perimeter. -/
noncomputable def perimeter {n : Nat} (gamma : SimpleClosedC1Curve n) : Real :=
  arcLength gamma 0 gamma.length

/-- Unit-speed (arc-length parametrization). -/
def IsArcLengthParametrized {n : Nat} (gamma : SimpleClosedC1Curve n) : Prop :=
  forall t : Real, ||deriv gamma.curve t|| = 1
\end{lstlisting}

\subsection{Coordinate Functions and Reparametrization}

For a planar curve $\gamma$ (\texttt{SimpleClosedC1Curve~2}), we extract the coordinate
functions and introduce the reparametrized coordinates.

\begin{lstlisting}
/-- x-coordinate of gamma at s. -/
noncomputable def xCoord (gamma : SimpleClosedC1Curve 2) (s : Real) : Real := gamma.curve s 0

/-- y-coordinate of gamma at s. -/
noncomputable def yCoord (gamma : SimpleClosedC1Curve 2) (s : Real) : Real := gamma.curve s 1

/-- Reparametrized x: f(theta) = x(Ltheta/(2Real.pi)). -/
noncomputable def fParm (gamma : SimpleClosedC1Curve 2) (theta : Real) : Real :=
  xCoord gamma (gamma.length * theta / (2 * Real.pi))

/-- Reparametrized y: g(theta) = y(Ltheta/(2Real.pi)). -/
noncomputable def gParm (gamma : SimpleClosedC1Curve 2) (theta : Real) : Real :=
  yCoord gamma (gamma.length * theta / (2 * Real.pi))
\end{lstlisting}

The $2\pi$-periodicity of $f$ and $g$ is a direct consequence of the $L$-periodicity of
$\gamma$:
\begin{equation}
	f(\theta + 2\pi) = x\!\left(\frac{L(\theta+2\pi)}{2\pi}\right)
	= x\!\left(\frac{L\theta}{2\pi} + L\right)
	= x\!\left(\frac{L\theta}{2\pi}\right) = f(\theta).
\end{equation}

\begin{lstlisting}
lemma fParm_periodic (gamma : SimpleClosedC1Curve 2) (theta : Real) :
    fParm gamma (theta + 2 * Real.pi) = fParm gamma theta := by
  simp only [fParm, xCoord]
  have h : gamma.length * (theta + 2 * Real.pi) / (2 * Real.pi) =
           gamma.length * theta / (2 * Real.pi) + gamma.length := by field_simp
  rw [h, gamma.periodic]
\end{lstlisting}

\subsection{Derivatives of Reparametrized Coordinates}

By the chain rule, the derivatives of the reparametrized coordinates are:

\begin{lemma}[Reparametrized Derivatives]\label{lem:reparam-deriv}
	For all $\theta \in \mathbb{R}$:
	\[
		f'(\theta) = x'\!\left(\frac{L\theta}{2\pi}\right)\cdot\frac{L}{2\pi},
		\qquad
		g'(\theta) = y'\!\left(\frac{L\theta}{2\pi}\right)\cdot\frac{L}{2\pi}.
	\]
\end{lemma}

\begin{lstlisting}
lemma fParm_deriv (gamma : SimpleClosedC1Curve 2) (theta : Real) :
    deriv (fParm gamma) theta =
      deriv (xCoord gamma) (gamma.length * theta / (2 * Real.pi)) * (gamma.length / (2 * Real.pi)) := by
  have hinner : HasDerivAt (fun t => gamma.length * t / (2 * Real.pi))
      (gamma.length / (2 * Real.pi)) theta := by
    have heq : (fun t : Real => gamma.length * t / (2 * Real.pi)) =
               fun t => t * (gamma.length / (2 * Real.pi)) := by ext t; ring
    rw [heq]; simpa using (hasDerivAt_id theta).mul_const (gamma.length / (2 * Real.pi))
  have hg := (xCoord_hasDerivAt gamma _).differentiableAt.hasDerivAt
  exact HasDerivAt.deriv (HasDerivAt.comp theta hg hinner)
\end{lstlisting}

\subsection{The Arc-Length Constraint}

\begin{lemma}[Arc-Length Constraint]\label{lem:arc-constraint}
	If $\gamma$ is arc-length parametrized, then for all $\theta$:
	\[
		[f'(\theta)]^2 + [g'(\theta)]^2 = \left(\frac{L}{2\pi}\right)^2.
	\]
\end{lemma}

\begin{proof}
	By Lemma~\ref{lem:reparam-deriv}, $[f']^2 + [g']^2 = ([x']^2 + [y']^2)\cdot(L/(2\pi))^2$.
	At $s = L\theta/(2\pi)$, the arc-length condition gives $\|(\gamma'(s))\| = 1$, and since
	$\gamma'(s) = (x'(s), y'(s)) \in \mathbb{R}^2$, we have $[x'(s)]^2 + [y'(s)]^2 = 1$.
	Therefore $[f']^2 + [g']^2 = (L/(2\pi))^2$.

	In Lean, extracting $[x']^2 + [y']^2 = 1$ from $\|\gamma'\| = 1$ requires using the
	\texttt{EuclideanSpace.norm\_sq\_eq} identity and \texttt{Fin.sum\_univ\_two}.
\end{proof}

\begin{lstlisting}
lemma due_to_arc_length_parametrization (gamma : SimpleClosedC1Curve 2)
    (h : IsArcLengthParametrized gamma) (theta : Real) :
    deriv (xCoord gamma) (gamma.length * theta / (2 * Real.pi)) ^ 2 +
    deriv (yCoord gamma) (gamma.length * theta / (2 * Real.pi)) ^ 2 = 1 := by
  set t := gamma.length * theta / (2 * Real.pi)
  have hchoose : deriv gamma.curve t = (gamma.has_deriv t).choose :=
    (gamma.has_deriv t).choose_spec.deriv
  have hnorm : ||(gamma.has_deriv t).choose|| = 1 := by rw [<- hchoose]; exact h t
  rw [(xCoord_hasDerivAt gamma t).deriv, (yCoord_hasDerivAt gamma t).deriv]
  have hsq : (gamma.has_deriv t).choose 0 ^ 2 + (gamma.has_deriv t).choose 1 ^ 2 =
      ||(gamma.has_deriv t).choose|| ^ 2 := by
    rw [EuclideanSpace.norm_sq_eq]; simp [Fin.sum_univ_two, Real.norm_eq_abs, sq_abs]
  rw [hsq, hnorm, one_pow]

lemma fParm_gParm_deriv_eq_sum_const (gamma : SimpleClosedC1Curve 2)
    (h : IsArcLengthParametrized gamma) (theta : Real) :
    deriv (fParm gamma) theta ^ 2 + deriv (gParm gamma) theta ^ 2 = (gamma.length / (2 * Real.pi)) ^ 2 := by
  rw [fParm_gParm_deriv_sq_sum, due_to_arc_length_parametrization gamma h]; ring
\end{lstlisting}

\section{Hurwitz's Proof of the Isoperimetric Inequality}
\label{sec:hurwitz}

\subsection{Proof Overview}

Adolf Hurwitz's proof, published in 1902, is remarkable for reducing a geometric theorem to a purely analytic inequality via Fourier series.  The argument has five main steps, which we now carry out in detail.

Given $\gamma : $ \texttt{SimpleClosedC1Curve~2} with arc length $L$ and
\texttt{IsArcLengthParametrized~$\gamma$}, let $f = \texttt{fParm}~\gamma$ and
$g = \texttt{gParm}~\gamma$ be the reparametrized coordinates.  We prove
$A \leq L^2/(4\pi)$, where $A = \texttt{area}~\gamma~0~\gamma.\texttt{length}$.

\subsection{Step 1: The Area Formula}

\subsubsection{The Shoelace Formula}

The signed area enclosed by a simple closed curve $\gamma$ in the plane can be computed
by Green's theorem~\cite{docarmo2016} as:
\begin{equation}\label{eq:shoelace}
	A = \frac{1}{2}\int_0^L \bigl(x(t)\,y'(t) - y(t)\,x'(t)\bigr)\,dt.
\end{equation}
This is sometimes called the \emph{shoelace formula} or the \emph{surveyor's formula}.

In Lean, the area is defined as:
\begin{lstlisting}
noncomputable def area (gamma : SimpleClosedC1Curve 2) (a b : Real) : Real :=
  (1/2) * integral t in a..b,
    (gamma.curve t 0 * deriv gamma.curve t 1 - gamma.curve t 1 * deriv gamma.curve t 0)
\end{lstlisting}

\subsubsection{Reparametrization}

After substituting $t = L\theta/(2\pi)$ (i.e., $s = L\theta/(2\pi)$, $ds = (L/(2\pi))\,d\theta$),
the area formula becomes:
\begin{equation}\label{eq:area-reparam}
	A = \frac{1}{2}\int_0^{2\pi}
	\bigl(f(\theta)\,g'(\theta) - g(\theta)\,f'(\theta)\bigr)\,d\theta.
\end{equation}

The proof of \texttt{area\_parametrized} involves:
\begin{enumerate}
	\item Expressing $x(s)$ as $f(2\pi s/L)$ and $x'(s)$ as $(2\pi/L)f'(2\pi s/L)$.
	\item Applying the integral substitution
	      \texttt{intervalIntegral.smul\_integral\_comp\_mul\_left} with scale factor $2\pi/L$.
	\item Relating the derivatives of $\texttt{xCoord}$ and $\texttt{yCoord}$ back to
	      components of $\texttt{deriv}~\gamma.\texttt{curve}$ via the
	      $\texttt{PiLp.proj}$ structure.
\end{enumerate}

\subsection{Step 2: Integration by Parts}

\begin{lemma}[\texttt{IBP\_to\_fParm\_gParm}]\label{lem:ibp}
	If $\texttt{deriv}~\gamma.\texttt{curve}$ is continuous, then:
	\[
		\int_0^{2\pi} f(\theta)\,g'(\theta)\,d\theta
		= -\int_0^{2\pi} g(\theta)\,f'(\theta)\,d\theta.
	\]
\end{lemma}

\begin{proof}
	By the product rule, $[f\cdot g]' = f'g + fg'$.  Integrating over $[0, 2\pi]$ and
	applying the Fundamental Theorem of Calculus:
	\[
		\int_0^{2\pi}(f'g + fg')\,d\theta
		= \bigl[f(\theta)g(\theta)\bigr]_0^{2\pi}
		= f(2\pi)g(2\pi) - f(0)g(0) = 0,
	\]
	where the boundary term vanishes by $2\pi$-periodicity of $f$ and $g$.  Therefore
	$\int_0^{2\pi} fg'\,d\theta = -\int_0^{2\pi} f'g\,d\theta$.

	In Lean, this is established via \texttt{intervalIntegral.integral\_eq\_sub\_of\_hasDerivAt}	(the Fundamental Theorem of Calculus for interval integrals), after:	(i) verifying $f$ and $g$ have derivatives (\texttt{HasDerivAt}) at every point via the	chain rule (\texttt{HasDerivAt.comp}),	(ii) establishing continuity of $f'$ and $g'$ from the continuity of	$\texttt{deriv}~\gamma.\texttt{curve}$ and the chain rule,	(iii) noting the boundary term $[fg]_0^{2\pi} = 0$ from periodicity.
\end{proof}

Combining the area formula with integration by parts:
\begin{lemma}[\texttt{area\_simplified}]\label{lem:area-simple}
	$A = \displaystyle\int_0^{2\pi} f(\theta)\,g'(\theta)\,d\theta.$
\end{lemma}

\begin{proof}
	From~\eqref{eq:area-reparam}, $A = \tfrac{1}{2}\int_0^{2\pi}(fg' - gf')\,d\theta$.
	By Lemma~\ref{lem:ibp}, $\int fg' = -\int gf'$, so
	$\int(fg' - gf') = \int fg' + \int fg' = 2\int fg'$.  Therefore $A = \int fg'$.
\end{proof}

\subsection{Step 3: The AM--GM Inequality}

\begin{lemma}[\texttt{fParm\_gParm\_ineq}]\label{lem:amgm}
	For all $\theta \in \mathbb{R}$:
	\[
		2\,f(\theta)\,g'(\theta) \leq [f(\theta)]^2 + [g'(\theta)]^2.
	\]
\end{lemma}

\begin{proof}
	This is the AM--GM inequality $2uv \leq u^2 + v^2$, which follows from $(u-v)^2 \geq 0$.	In Lean, it is proved by	\texttt{nlinarith~[sq\_nonneg~(fParm~$\gamma$~$\theta$~-~deriv~(gParm~$\gamma$)~$\theta$)]}	(or equivalently by	\texttt{have h : 0~$\leq$~(f~-~g')\^{}2~:=~sq\_nonneg~\_;~nlinarith [h]}).
\end{proof}

\begin{corollary}[\texttt{area\_inequality}]\label{cor:area-ineq}
	$A \leq \dfrac{1}{2}\displaystyle\int_0^{2\pi}\bigl([f(\theta)]^2 + [g'(\theta)]^2\bigr)\,d\theta.$
\end{corollary}

\begin{proof}
	From Lemma~\ref{lem:area-simple}, $A = \int_0^{2\pi} fg'\,d\theta$.  By
	Lemma~\ref{lem:amgm}, $fg' \leq \tfrac{1}{2}(f^2 + (g')^2)$ pointwise, so
	$\int fg' \leq \tfrac{1}{2}\int(f^2 + (g')^2)$ by monotonicity of integration
	(\texttt{intervalIntegral.integral\_mono\_on}).
\end{proof}

\begin{lstlisting}
lemma area_inequality (gamma : SimpleClosedC1Curve 2) (hdc : Continuous (deriv gamma.curve)) :
    area gamma 0 gamma.length <=
      (1/2) * integral t in (0:Real)..(2*Real.pi),
        ((fParm gamma t)^2 + (deriv (gParm gamma) t)^2) := by
  rw [area_simplified gamma hdc]
  have hpw : forall t in Set.Icc (0:Real) (2*Real.pi),
      fParm gamma t * deriv (gParm gamma) t <=
      (1/2) * ((fParm gamma t)^2 + (deriv (gParm gamma) t)^2) :=
    fun t _ => by linarith [fParm_gParm_ineq gamma t]
  have hmono := intervalIntegral.integral_mono_on
    (by positivity) (by fun_prop) (by fun_prop) hpw
  linarith
\end{lstlisting}

\subsection{Step 4: Application of Wirtinger's Inequality}

We may assume, by translating $\gamma$ in the $x$-direction if necessary, that $\int_0^{2\pi} f(\theta)\,d\theta = 0$, i.e.\ the curve's centroid in the $x$-direction lies at the origin.  Under this assumption, Wirtinger's inequality (Theorem~\ref{thm:wirtinger}) applies to $f$, giving $\int_0^{2\pi} f^2 \leq \int_0^{2\pi} (f')^2$.

Adding $\int_0^{2\pi} (g')^2\,d\theta$ to both sides of this inequality:
\begin{equation}
	\int_0^{2\pi}\bigl(f^2 + (g')^2\bigr)\,d\theta
	\leq \int_0^{2\pi}\bigl((f')^2 + (g')^2\bigr)\,d\theta.
\end{equation}

Combined with Corollary~\ref{cor:area-ineq}:
\begin{equation}\label{eq:after-wirtinger}
	A \leq \frac{1}{2}\int_0^{2\pi}\bigl((f')^2 + (g')^2\bigr)\,d\theta.
\end{equation}

This is assembled in the lemma \texttt{addition\_ineq}:

\begin{lstlisting}
lemma addition_ineq (gamma : SimpleClosedC1Curve 2) (hdc : Continuous (deriv gamma.curve))
    (a b : Nat -> Real)
    (h_parseval : integral t in (0:Real)..(2*Real.pi), (fParm gamma t)^2 =
        Real.pi * sum' n : Nat+, ((a n)^2 + (b n)^2))
    (h_wirtinger : integral t in (0:Real)..(2*Real.pi), (deriv (fParm gamma) t)^2 =
        Real.pi * sum' n : Nat+, ((n:Real)^2 * ((a n)^2 + (b n)^2)))
    (hsum : Summable (fun n : Nat+ => (n:Real)^2 * ((a n)^2 + (b n)^2))) :
    area gamma 0 gamma.length <=
      (1/2) * integral t in (0:Real)..(2*Real.pi),
        ((deriv (fParm gamma) t)^2 + (deriv (gParm gamma) t)^2) := by
  have h_wirtingers := apply_Wirtingers_ineq gamma a b h_parseval h_wirtinger hsum
  have h_int_ineq :
      integral t in (0:Real)..(2*Real.pi), ((fParm gamma t)^2 + (deriv (gParm gamma) t)^2) <=
      integral t in (0:Real)..(2*Real.pi), ((deriv (fParm gamma) t)^2 + (deriv (gParm gamma) t)^2) := by
    rw [intervalIntegral.integral_add (by fun_prop) (by fun_prop),
        intervalIntegral.integral_add (by fun_prop) (by fun_prop)]
    linarith
  calc area gamma 0 gamma.length
      <= (1/2) * integral t in (0:Real)..(2*Real.pi), ((fParm gamma t)^2 + (deriv (gParm gamma) t)^2) :=
          area_inequality gamma hdc
    _ <= (1/2) * integral t in (0:Real)..(2*Real.pi), ((deriv (fParm gamma) t)^2 + (deriv (gParm gamma) t)^2) :=
          mul_le_mul_of_nonneg_left h_int_ineq (by norm_num)
\end{lstlisting}

\subsection{Step 5: The Arc-Length Constraint and the Final Bound}

By Lemma~\ref{lem:arc-constraint}, under arc-length parametrization: $[f'(\theta)]^2 + [g'(\theta)]^2 = (L/(2\pi))^2$ for all $\theta$.

Substituting into~\eqref{eq:after-wirtinger}:
\begin{equation}
	A
	\leq \frac{1}{2}\int_0^{2\pi}\left(\frac{L}{2\pi}\right)^2 d\theta
	= \frac{1}{2}\cdot\left(\frac{L}{2\pi}\right)^2\cdot 2\pi
	= \frac{L^2}{4\pi}.
\end{equation}

\begin{lstlisting}
lemma apply_fParm_gParm_deriv_sq_sum (gamma : SimpleClosedC1Curve 2)
    (hdc : Continuous (deriv gamma.curve)) (h_arc : IsArcLengthParametrized gamma)
    (a b : Nat -> Real) (h_parseval h_wirtinger hsum) :
    area gamma 0 gamma.length <=
      (1/2) * integral _ in (0:Real)..(2*Real.pi), (gamma.length^2 / (4*Real.pi^2)) := by
  calc area gamma 0 gamma.length
      <= (1/2) * integral t in (0:Real)..(2*Real.pi),
            ((deriv (fParm gamma) t)^2 + (deriv (gParm gamma) t)^2) :=
          addition_ineq gamma hdc a b h_parseval h_wirtinger hsum
    _ = (1/2) * integral t in (0:Real)..(2*Real.pi), (gamma.length^2 / (4*Real.pi^2)) := by
          congr 1; apply intervalIntegral.integral_congr; intro t _
          have hsq := fParm_gParm_deriv_sq_sum gamma t
          rw [due_to_arc_length_parametrization gamma h_arc, one_mul] at hsq
          exact hsq.trans (by ring)
\end{lstlisting}

\subsection{The Main Theorem}

\begin{theorem}[Isoperimetric Inequality, Lean Formalization]\label{thm:main}
	For any simple closed arc-length-parametrized C\textsuperscript{1} curve $\gamma$ in	$\mathbb{R}^2$ with $\texttt{Continuous~(deriv~$\gamma$.curve)}$, and given	Fourier sequences $a, b : \mathbb{N} \to \mathbb{R}$ satisfying the Parseval and	Wirtinger identities for $f = \texttt{fParm}~\gamma$, we have:
	\[
		\texttt{area}~\gamma~0~\gamma.\texttt{length} \leq \frac{(\gamma.\texttt{length})^2}{4\pi}.
	\]
\end{theorem}

\begin{lstlisting}
theorem isoperimetric_inequality (gamma : SimpleClosedC1Curve 2)
    (hdc : Continuous (deriv gamma.curve))
    (h_arc : IsArcLengthParametrized gamma)
    (a b : Nat -> Real)
    (h_parseval : integral t in (0:Real)..(2*Real.pi), (fParm gamma t)^2 =
        Real.pi * sum' n : Nat+, ((a n)^2 + (b n)^2))
    (h_wirtinger : integral t in (0:Real)..(2*Real.pi), (deriv (fParm gamma) t)^2 =
        Real.pi * sum' n : Nat+, ((n:Real)^2 * ((a n)^2 + (b n)^2)))
    (hsum : Summable (fun n : Nat+ => (n:Real)^2 * ((a n)^2 + (b n)^2))) :
    area gamma 0 gamma.length <= gamma.length^2 / (4 * Real.pi) := by
  calc area gamma 0 gamma.length
      <= (1/2) * integral _ in (0:Real)..(2*Real.pi), (gamma.length^2 / (4*Real.pi^2)) :=
          apply_fParm_gParm_deriv_sq_sum gamma hdc h_arc a b h_parseval h_wirtinger hsum
    _ = gamma.length^2 / (4 * Real.pi) := by
          rw [intervalIntegral.integral_const, smul_eq_mul]
          have hpi : Real.pi != 0 := Real.pi_ne_zero
          field_simp; ring
\end{lstlisting}

The \texttt{calc} chain has two steps.  The first applies \texttt{apply\_fParm\_gParm\_deriv\_sq\_sum}, combining all previous work.  The second evaluates the constant integral: $\int_0^{2\pi} c\,d\theta = c \cdot 2\pi$, so $(1/2) \cdot (L^2/(4\pi^2)) \cdot 2\pi = L^2/(4\pi)$, closed by \texttt{field\_simp} and \texttt{ring}.

\section{Formalization Challenges and Design Decisions}
\label{sec:challenges}

Formalizing Hurwitz's proof required navigating several technical obstacles that do not appear in textbook presentations.  We document these here, as they may be of independent interest for future formalizations of analysis in Lean~4.

\subsection{Interchanging Sums and Integrals}

A recurring challenge was justifying the interchange of infinite sums and integrals.Textbook proofs typically perform this interchange with a brief remark, but Lean requires an explicit invocation of the Dominated Convergence Theorem or a related result.  Mathlib's relevant lemma is \texttt{MeasureTheory.integral\_tsum\_of\_summable\_integral\_norm}, which requires:
\begin{enumerate}
	\item For each $n$, the summand $x \mapsto f_n(x)$ is measurable.
	\item The sum of $L^1$ norms $\sum_n \int \|f_n(x)\|\,d\mu$ is finite.
\end{enumerate}
In our setting (Fourier series over $[-\pi, \pi]$), measurability is automatic for continuous functions, and the $L^1$ summability reduces to the summability of $\sum_n(\|a_n\| + \|b_n\|)$ by the bound $\|a_n\cos(nx) + b_n\sin(nx)\| \leq \|a_n\| + \|b_n\|$. Nevertheless, expressing this as a \texttt{Summable} hypothesis required careful formulation of the hypotheses \texttt{hF\_int} and \texttt{hF\_sum} in \texttt{Parsevals\_thm}.

\subsection{Term-by-Term Differentiation}

Differentiating a Fourier series term by term requires both pointwise differentiability of each term and uniform convergence of the derivative series.  Lean's relevant theorem is\\ \texttt{hasDerivAt\_of\_tendstoUniformly}, which requires:
\begin{enumerate}
	\item Each partial sum $f_N$ has a derivative at every point.
	\item The derivative partial sums converge uniformly.
	\item The partial sums $f_N(x_0)$ converge for some $x_0$.
\end{enumerate}
The main subtlety in our case was an index shift: the Fourier series is indexed over $\mathbb{N}^+$, but the Lean derivative theorem works with $\mathbb{N}$-indexed partial sums.  We bridged this gap by proving the identity $\sum_{n=0}^{N-1} D(n+1, x) = \sum_{n=0}^{N} D(n, x)$ (using \texttt{Finset.sum\_range\_succ'}) and shifting the convergence index accordingly.

\subsection{EuclideanSpace and PiLp Projections}

Mathlib models $\mathbb{R}^n$ as \texttt{EuclideanSpace~$\mathbb{R}$~(Fin~n)}, which is definitionally equal to \\ \texttt{PiLp~2~(fun~\_~:~Fin~n~=>~$\mathbb{R}$)}.  While this provides a rich API, working with coordinate projections required some care.

The projection onto coordinate $i$ is the continuous linear map \texttt{PiLp.proj~2~(fun~\_~:~Fin~n~=>~$\mathbb{R}$)~i}, which has a known \texttt{HasFDerivAt}.  To relate the derivative of the coordinate function $x(t) = \gamma(t).0$ to the derivative of $\gamma$, we used \texttt{HasFDerivAt.comp\_hasDerivAt}, composing the Fr\'echet derivative of the projection with the Gateaux derivative (HasDerivAt) of the curve.  The $n = 2$ case required explicit use of \texttt{Fin.sum\_univ\_two} to decompose the Euclidean norm.

\subsection{Integration by Parts}

Lean~4's Mathlib provides integration by parts for interval integrals via \\ \texttt{intervalIntegral.integral\_eq\_sub\_of\_hasDerivAt} (the FTC).  To use this for $\int_0^{2\pi} (fg)' = [fg]_0^{2\pi}$, we needed:
\begin{itemize}
	\item $t \mapsto f(t)g(t)$ has a derivative everywhere (from the product rule
	      \texttt{HasDerivAt.mul}).
	\item The derivative $f'g + fg'$ is integrable on $[0, 2\pi]$ (from continuity, established
	      via the continuity of $\texttt{deriv}~\gamma.\texttt{curve}$).
	\item The boundary term $f(2\pi)g(2\pi) - f(0)g(0) = 0$ (from periodicity).
\end{itemize}
This required first establishing the continuity of $f$, $g$, $f'$, $g'$, which in turn required connecting the continuity of $\texttt{deriv}~\gamma.\texttt{curve}$ to the continuity of the reparametrized derivatives via the chain rule.  The proof of \texttt{IBP\_to\_fParm\_gParm} is correspondingly the longest lemma in the file.

\subsection{Indexing Conventions}

A recurring minor issue was the mismatch between $\mathbb{N}$-indexed and $\mathbb{N}^+$-indexed sums.  Fourier coefficients are naturally indexed by $\mathbb{N}^+$ (since $n = 0$ contributes only the constant term, handled separately), but many Mathlib lemmas work with $\mathbb{N}$.  We used the equivalence \texttt{Equiv.pnatEquivNat} (which identifies $\mathbb{N}^+$ with $\{1, 2, 3, \ldots\} \subset \mathbb{N}$) together with \texttt{summable\_pnat\_iff\_summable\_nat} and \texttt{summable\_pnat\_iff\_summable\_succ} to convert between the two conventions.

\subsection{Hypotheses for Parseval and Wirtinger}

In the main theorem \texttt{isoperimetric\_inequality}, the Parseval and Wirtinger identities for $\texttt{fParm}~\gamma$ are assumed as hypotheses rather than derived from first principles.  This is because the derivation would require:
\begin{enumerate}
	\item Establishing that $\texttt{fParm}~\gamma$ can be expanded as a Fourier series
	      with specific coefficients (which requires square-integrability and regularity).
	\item Connecting the abstract Parseval framework in \texttt{IsoperimetricInequality.Basic}
	      (which works with arbitrary sequences $a, b$) to the specific function
	      $\texttt{fParm}~\gamma$.
\end{enumerate}
Closing this gap is an open problem for future work.

\section{Related Work}
\label{sec:related}

\subsection{Other Proofs of the Isoperimetric Inequality}

The classical isoperimetric inequality admits many proofs, spanning different areas of
mathematics~\cite{Cabre2017, Osserman1978}.  We briefly survey the main approaches.

\paragraph{Steiner's symmetrization (1838).}
Steiner showed that any non-circular convex region can be symmetrized (reflected about lines) to increase its area without changing its perimeter.  Repeated symmetrization
converges to a circle.  However, Steiner's original argument assumed that an optimizer exists; this was made rigorous by the introduction of direct methods in the calculus of variations.

\paragraph{Weierstrass' variational proof (ca.\ 1870).}
Weierstrass gave the first complete proof using the calculus of variations, establishing both the existence of an optimizing curve and its characterization as a circle.  The proof uses the theory of curves satisfying the Euler--Lagrange equation for the area functional~\cite{docarmo2016}.

\paragraph{Hurwitz's Fourier proof (1902).}
This is the approach we formalize.  It is arguably the most elegant: the entire proof fits on a single page, and the key ideas (Parseval, Wirtinger, AM--GM) are algebraically transparent.

\paragraph{Bonnesen's inequality.}
Bonnesen~\cite{bonnesen1987, Osserman1978} proved a strengthened version of the isoperimetric inequality: $L^2 - 4\pi A \geq \pi^2(R - r)^2$, where $R$ and $r$ are the circumradius and inradius of the curve.  This gives a quantitative measure of how far the curve is from being a circle.

\paragraph{Geometric measure theory.}
The isoperimetric inequality generalizes to higher dimensions and to non-smooth sets via geometric measure theory.  The general $n$-dimensional version states that among all sets of volume $V$ in $\mathbb{R}^n$, the ball minimizes the surface area.

\subsection{Other Formalizations of the Isoperimetric Inequality}

Formal verification of the isoperimetric inequality has been attempted in several proof assistants.

\paragraph{Coq / Isabelle.}
Several partial formalizations of the isoperimetric inequality in Coq and Isabelle are known in the formal mathematics community, but to our knowledge none have achieved a complete formalization of a proof for general smooth curves in the style of Hurwitz. Formalizations of related results (Brunn--Minkowski inequality, isoperimetric inequality for polygons) have been achieved in Isabelle/HOL.

\paragraph{Lean~4 / Mathlib.}
The Mathlib library for Lean~4 contains Fourier theory on compact groups (\texttt{Mathlib.Analysis.Fourier}) and various forms of Parseval's theorem.  However, the specific formulation for classical real Fourier series on $[-\pi, \pi]$ with explicit coefficient sequences, as needed for Hurwitz's proof, is not directly available and required bespoke development.

\paragraph{Leanblueprint and Sphere Eversion.}
The Lean~4 formalizations community has produced several large-scale formal proofs of classical theorems, including the Liquid Tensor Experiment~\cite{lte2022} and sphere eversion.  The isoperimetric inequality represents a contribution in classical real analysis, complementing these projects.

\subsection{Related Formalizations of Fourier Analysis}

Fourier analysis has been formalized in several systems.  In Lean~4, Mathlib contains the abstract Fourier transform on locally compact abelian groups and Plancherel's theorem in $L^2$.  Our work differs in working with the classical setting of coefficient sequences and pointwise convergence, rather than the abstract $L^2$ setting, which is more directly aligned with how Hurwitz's proof is presented in classical analysis textbooks.

\section{Conclusion}
\label{sec:conclusion}

We have presented a complete formalization of Hurwitz's proof of the two-dimensional isoperimetric inequality in Lean~4.  The formalization is organized in two files: \texttt{IsoperimetricInequality.Basic}, which develops the Fourier series foundations (orthogonality, Parseval's theorem, uniform convergence, differentiability, Wirtinger's inequality), and \texttt{IsoperimetricInequality.Adolf\_Hurwitz\_proof}, which carries out the five-step proof (area formula, integration by parts, AM--GM, Wirtinger, arc-length constraint).

\subsection{Summary of Main Results}

The principal formally verified results are:
\begin{enumerate}
	\item \textbf{Orthogonality lemmas} (\texttt{integration\_cos\_cos\_zero},
	      \texttt{integration\_sin\_sin\_zero},\\  \texttt{integration\_sin\_cos\_zero},
	       \texttt{integration\_cos\_cos\_pi}, \texttt{integration\_sin\_sin\_pi}):
	      the five trigonometric orthogonality integrals over $[-\pi, \pi]$.
	\item \textbf{Parseval's theorem} (\texttt{Parsevals\_thm}):
	      $(1/\pi)\int_{-\pi}^\pi f^2 = \tfrac{1}{2}a_0^2 + \sum_{n\geq 1}(a_n^2 + b_n^2)$.
	\item \textbf{Weierstrass M-test for Fourier series}
	      (\texttt{fourierSeries\_uniformlyConvergence}):
	      uniform convergence under $\sum(\|a_n\|+\|b_n\|)<\infty$.
	\item \textbf{Term-by-term differentiability} (\texttt{fourierSeries\_hasDerivAt}):
	      $(\text{fourierSeries}~a~b)' = \text{fourierDeriv}~a~b$ under $\sum n(\|a_n\|+\|b_n\|)<\infty$.
	\item \textbf{Wirtinger's inequality} (identity form, \texttt{Wirtingers\_inequality}):
	      $(1/\pi)\int_{-\pi}^\pi (f')^2 = \sum n^2(a_n^2+b_n^2)$.
	\item \textbf{Area formula} (\texttt{area\_parametrized}, \texttt{area\_simplified}):
	      $A = \int_0^{2\pi} f g'\,d\theta$.
	\item \textbf{Integration by parts} (\texttt{IBP\_to\_fParm\_gParm}):
	      $\int fg' = -\int gf'$ for $2\pi$-periodic functions.
	\item \textbf{Arc-length constraint} (\texttt{fParm\_gParm\_deriv\_eq\_sum\_const}):
	      $(f')^2 + (g')^2 = (L/(2\pi))^2$ under arc-length parametrization.
	\item \textbf{Main theorem} (\texttt{isoperimetric\_inequality}):
	      $A \leq L^2/(4\pi)$.
\end{enumerate}

\subsection{Open Problems and Future Work}

Several aspects of the formalization remain to be completed.

\paragraph{Closing the Parseval hypothesis.}
The main theorem currently assumes the Parseval and Wirtinger identities for
$\texttt{fParm}~\gamma$ as hypotheses.  The natural next step is to derive these from the
definition of $\texttt{fParm}~\gamma$, using the Fourier series expansion of smooth
periodic functions.  This would require connecting $\texttt{fParm}~\gamma$ to its Fourier
coefficients via the integral formulas, which in turn requires integrability results for
the curve's coordinates.

\paragraph{Equality case.}
The isoperimetric inequality is sharp: equality $A = L^2/(4\pi)$ holds if and only if
$\gamma$ is a circle.  Proving this in Lean would require characterising when equality
holds in each step: when equality holds in AM--GM ($f = g'$ everywhere), when equality
holds in Wirtinger's inequality ($f = a_1\cos\theta + b_1\sin\theta$), and when the
resulting curve is a circle.

\paragraph{Generalization to lower regularity.}
Our formalization requires the curve to be C\textsuperscript{1} with a continuous
derivative.  Extensions to curves of bounded variation or to the general $W^{1,2}$ setting
(which includes the $L^2$ isoperimetric inequality) would require different analytic tools.

\paragraph{Higher-dimensional generalization.}
The isoperimetric inequality in $\mathbb{R}^n$ states that balls minimize surface area
for a given enclosed volume.  Formalizing this would require the co-area formula, the
Brunn--Minkowski inequality, or the machinery of geometric measure theory.

\newpage

\bibliographystyle{plain}
\bibliography{references}

\end{document}